\newif\ifjournal
\journalfalse
\ifjournal
\documentclass{aims}
\else
\documentclass[12pt]{article}
\usepackage[top=1in, bottom=1in, left=1.25in, right=1.25in]{geometry}
\usepackage{times,amsfonts}
\date{\today}
\fi

\usepackage{algorithm2e}
\ifx\LinesNumbered\undefined\let\LinesNumbered\linesnumbered
\fi

\usepackage{amsmath}
\usepackage{amsthm}
\usepackage{graphicx}

\usepackage{color}
\definecolor{labelkey}{rgb}{0,0,.75}
\definecolor{MyGreen}{rgb}{0,.6,.2}
\definecolor{MyDarkBlue}{rgb}{.1,.1,.75}
\usepackage[colorlinks,citecolor=MyGreen, linkcolor=MyDarkBlue]{hyperref}

\newcommand{\abs}[1]{\left|#1\right|}

\newcommand{\ra}{\rightarrow}

\newcommand{\Reals}{\mathbb{R}}

\newif\ifeqstar
\newcommand{\be}{\begin{equation}}
\newcommand{\bel}[1]{\begin{equation}\label{#1}\eqstarfalse}
\newcommand{\ee}{\ifeqstar\end{equation*}\else\end{equation}\fi}
\newcommand{\bes}{\begin{equation*}\eqstartrue}

\newcommand{\calN}{\mathcal{N}}

\newcommand{\calD}{\mathcal{D}}

\def\ip<#1>{\left<#1\right>}

\DeclareMathOperator{\Laplacian}{\Delta}
\DeclareMathOperator{\Lip}{\mathrm{Lip}}

\let\Lap\Laplacian
\newcommand{\Hoo}{H^{1/2}_{00}}
\newcommand{\Hoos}{H^{s}_{00}}
\newcommand{\Hoons}{H^{-s}_{00}}
\newcommand{\GB}{B}
\newcommand{\GS}{S}
\newcommand{\pO}{\partial\Omega}

\newtheorem{proposition}{Proposition}

\newtheorem{lemma}{Lemma}
\newtheorem{remark}{Remark}

\begin{document}
\author{David Maxwell}
\title{Kozlov-Maz'ya iteration as a form of Landweber iteration}
\maketitle
\begin{abstract}We consider the alternating method of Kozlov and Maz'ya for solving the Cauchy problem for elliptic boundary-value problems.  Considering the case of the Laplacian, we show that this method can be recast as a form of Landweber iteration.  In addition to conceptual advantages, this observation leads to some practical improvements.  We show how to accelerate Kozlov-Maz'ya iteration using the conjugate gradient algorithm, and we show how to modify the method to obtain a more practical stopping criterion.
\end{abstract}
\section{Introduction}\label{sec:intro}
The Cauchy problem for elliptic equations, where Dirichlet and Neumann data are prescribed simultaneously on a strict subset of the domain
boundary, is a prototypical ill-posed problem.  It is a linear problem, and can be approached
using any of a number of standard regularization techniques such as Tikhonov regularization \cite{CimetiereCauchy}, 
as well as logarithmic convexity methods\cite{PayneImproperlyPosed}.  
Kozlov and Maz'ya \cite{KozlovMazya1990} (see also \cite{KozlovMazyaFomin}) introduced 
a novel method for solving this problem that, 
while related to the general class of iterative methods, was not evidently one of the standard ones.  The method 
(sometimes known as the alternating method and called here Kozlov-Maz'ya iteration) 
is straightforward to implement numerically, and is therefore an attractive choice for practical use.  There are some
drawbacks, however.  The formal stopping criterion for Kozlov-Maz'ya iteration involves error estimates in certain Sobolov spaces with fractional derivatives.  Arriving at such  estimates for real data poses some difficulty.  Moreover,
it has been observed that Kozlov-Maz'ya iteration suffers from being slow.  Although there have been efforts to accelerate
the method using certain relaxation factors \cite{JourhamaneNatchaoui} \cite{JourhmaneLesnicMera}, formal proofs of the
stability of these ad-hoc techniques are not available.  

Our primary result is a demonstration that Kozlov-Maz'ya iteration is, in fact, a form of Landweber iteration \cite{Landweber1951} between function spaces equipped with suitable norms.  This observation yields a number of advantages.  First, the extensive  body of literature concerning Landweber iteration can
be brought to apply to Kozlov-Maz'ya iteration.  Proofs concerning its stability and rate of convergence can then be quoted from textbooks.
Second, standard techniques for accelerating Landweber iteration can be applied to Kozlov-Maz'ya iteration.  We indicate here a variation
based on the conjugate gradient method that is nearly as simple as standard Kozlov-Maz'ya iteration and
leads to a fast, order-optimal regularization method. Finally, we show how to 
modify one of the function spaces involved in the Landweber and conjugate gradient iterations to obtain a similar method with 
a more practical stopping criterion.

The motivation for this work comes from an inverse problem in glaciology \cite{MaxwellTrufferAvdonin}, which considered the Cauchy problem for a nonlinear elliptic PDE.  In that paper, the linearized inverse problems were solved using Kozlov-Mazya iteration, accelerated by the techniques described below in Section \ref{sec:cg}.  For simplicity, we focus our attention here on a model elliptic problem for the Laplacian; the extension to more general elliptic operators is straightforward. We remark that Kozlov-Maz'ya iteration has been extended to 
certain parabolic and hyperbolic inverse problems \cite{BaumeisterLeitao} 
as well as to a degenerate elliptic problem (the Stokes system) \cite{BastayEtAl2005}.  We do not treat these 
problems, but it hoped that the ideas presented here might also be useful in these cases.

\subsection{Formulation of the model Cauchy problem}
Let $\Omega\subseteq\Reals^n$ be an open, bounded, and connected set with a Lipschitz boundary $\pO$. Suppose
$\GS$ and $\GB$ are nonempty open subsets of $\pO$ sharing a common
boundary $\Pi$ and that $\pO=\GS\cup \GB\cup \Pi$ is a Lipschitz dissection as defined in \cite{McleanElliptic} (effectively $\Pi$ is an embedded Lipschitz hypersurface of $\Gamma$).
We suppose that boundary data is known on $\GS$ but unknown
on $\GB$; the notation suggests that $\GS$ is an accessible surface and $\GB$ is an inaccessible base.  
The Cauchy problem for the Laplacian is the following:
\bel{cauchy}	
\begin{aligned}
-\Laplacian	u &= f &&\text{in $\Omega$}\\
u &= \sigma &&\text{on $\GS$}\\
\partial_n u &= \tau &&\text{on $\GS$}.\\
\end{aligned}
\ee
Here $f\in (H^{-1}(\Omega))^*$, $\sigma\in H^{1/2}(\GS)$, $\tau\in H^{-1/2}(\GS)$, and $\partial_n$ denotes
the the normal derivative.  

Our notation and conventions for Sobolov spaces follow \cite{McleanElliptic},
except for one case noted below. The space $H^{s}(\GS)$ is the set of restrictions of distributions in $H^{s}(\partial\Omega)$ to $\GS$ and has the quotient norm; in this paper we will only use the cases $s=-1/2$, $0$, $1/2$. The subset of distributions $\sigma\in H^s(\partial\Omega)$ such that $\sigma|_B=0$ is denoted
by $H^{s}_{00}(S)$ (and by $\widetilde H^s(S)$ in \cite{McleanElliptic}).  It is a 
closed subspace of $H^{s}(\partial\Omega)$ and inherits the norm from the larger space.  We will consider elements of $H^{s}_{00}(S)$ as elements of $H^s(S)$ or as elements of $H^s(\partial \Omega)$ interchangeably 
and without comment.  Because of the regularity of the sets $\GS$ and $\GB$, there is a natural identification of $H^{-1/2}(\GS)$ with the dual space of $H^{1/2}_{00}(\GS)$.  More details concerning these conventions can be found in the Appendix.

The key step of recasting Kozlov-Maz'ya iteration as Landweber iteration requires making a judicious choice of (equivalent) norms on these boundary Sobolov spaces so that the adjoints of certain operators take on a natural form.  Proofs of the equivalence of these norms can be found in the Appendix.
For brevity we describe these norms here for distributions on $\GS$ with obvious adjustments needed for distributions on $\GB$.

\subsubsection{The space $H^{-1/2}(\GS)$}
Let $\psi\in H^{-1/2}(S)$ and let $v_\psi$ be the solution of
\bes
\begin{aligned}
\Laplacian	v_\psi &= 0 &&\text{in $\Omega$}\\
\partial_n v_\psi &= \psi &&\text{on $\GS$}\\
v_\psi &= 0 &&\text{on $\GB$}.
\end{aligned}
\ee
Then
\bes
||\psi||_{H^{-1/2}(\GS)}^2 = \int_\Omega \abs{\nabla v_\psi}^2.
\ee
Similarly, if $\psi_1, \psi_2\in H^{-1/2}(\GS)$ then
\bes
\ip<\psi_1,\psi_2>_{H^{-1/2}(\GS)} = \int_\Omega \nabla v_{\psi_1}\cdot
\nabla v_{\psi_2}.
\ee
This norm was described in the original work on Kozlov-Maz'ya iteration \cite{KozlovMazya1990}.

\subsubsection{The space $\Hoo(\GS)$}
Let $\phi\in \Hoo(\GS)$, so $\phi\in H^{1/2}(\partial\Omega)$ and $\phi|_B=0$. 
Let $w_\phi$ be the solution of
\bes
\begin{aligned}
\Laplacian	w_\phi &= 0 &&\text{in $\Omega$}\\
w_\phi &= \phi &&\text{on $\pO$}.
\end{aligned}
\ee
Then
\bes
||\phi||_{H^{1/2}(\GS)}^2 = \int_\Omega \abs{\nabla w_\phi}^2.
\ee
and there is a corresponding inner product
\bes
\ip<\phi_1,\phi_2>_{\Hoo(S)} 
= \int_\Omega \nabla w_{\phi_1}\cdot \nabla w_{\phi_2}.
\ee

\subsubsection{The space $H^{1/2}(\GS)$}\label{sec:H12}
Let $\phi\in H^{1/2}(\GS)$.
Let $r_\phi$ be the solution of
\bes
\begin{aligned}
	\Laplacian	r_\phi &= 0 &&\text{in $\Omega$}\\
	r_\phi &= \phi &&\text{on $\GS$}\\
	\partial_n r_\psi &= 0 &&\text{on $\GB$}.
\end{aligned}
\ee
Then
\bes
||\phi||_{H^{1/2}(\GS)}^2 = \int_\Omega \abs{\nabla r_\phi}^2 + \left[\frac{1}{|S|}\int_S r_\phi\right]^2,
\ee
where $|S|=\int_S 1$.  There is a corresponding inner product defined analogously to the other spaces.

\section{Kozlov-Maz'ya iteration}
Kozlov-Maz'ya iteration proceeds by alternating between solving boundary-value problems involving the Dirichlet data ($\sigma$) and Neumann data ($\tau$) in turn. Given $\psi\in H^{-1/2}(\GB)$, let $\calN(\psi)=v$ where $v\in H^1(\Omega)$ is the solution of
\bes
\begin{aligned}
-\Laplacian	v &= f &&\text{in $\Omega$}\\
v &= \sigma &&\text{on $\GS$}\\
\partial_{n} v &= \psi&&\text{on $\GB$}.
\end{aligned}
\ee
Given $\phi\in H^{1/2}(\GB)$, let $\calD(\phi)=w$ where $w\in H^1(\Omega)$
is the solution of 
\bes
\begin{aligned}
-\Laplacian w &= f &&\text{in $\Omega$}\\
\partial_n w &= \tau &&\text{on $\GS$}\\
w &= \phi &&\text{on $\GB$}.
\end{aligned}
\ee
The solvability of these equations is discussed briefly in the Appendix.

Starting with an initial element $\psi_0\in H^{-1/2}(B)$, we let
$v_0=\calN(\psi_0)$ and $w_0=\calD(v_0|_B)$.  Then $\psi_1=\partial_n w_0|_{\GB}$. Sequences $\{\psi_k\}$, $\{v_k\}$ and $\{w_k\}$ are obtained by repeating these operations.  Letting 
\bes
KM(\psi) = \partial_n\; \calD(\calN(\psi)|_B)\,|_B,
\ee 
we see that $\psi_{k+1}=KM(\psi_k)$.  Note that we use the convention
that operators with a script font yield distributions on $\Omega$,
whereas operators with a roman font yield distributions on a
subset of $\partial \Omega$.

It was proved in \cite{KozlovMazyaFomin} 
that $u\in H^1(\Omega)$ is a solution of the Cauchy problem \eqref{cauchy} if and only if $\partial_n u|_B$ is a fixed point of $KM$.   
Moreover, if there exists a solution $u$ of equation \eqref{cauchy}, then the functions
$v_n$ and $w_n$ converge to $u$ in $H^1(\Omega)$. And finally, 
for approximate boundary data $(\sigma_\delta,\tau_\delta)$, stopping the iteration early according to a discrepancy principle leads to a 
regularization  strategy for solving the Cauchy problem.

There is a dual formulation of Kozlov-Maz'ya iteration 
obtained via the operator
$MK:H^{1/2}(S)\ra H^{1/2}(S)$ defined by
\bes
MK(\phi) = \calN(\partial_n\calD(\phi)|_B)|_B
\ee
(i.e. $MK$ is $KM$ performed in the reverse order). In the following two sections we will show that the original form of Kozlov-Maz'ya iteration and its less-studied dual formulation can be exhibited as forms of Landweber iteration.  

The maps $\calD$, $\calN$, $KM$, and $MK$ are affine, and it will be useful
to have notation for their linear parts.
Let $\calD_0$, $\calN_0$, $KM_0$, and $MK_0$ be defined as before
but with homogeneous data ($f=0$, $\sigma=0$, $\tau=0$).
It then follows that for any $\psi,\hat\psi\in H^{-1/2}(\GB)$ and $\phi,\hat\phi\in H^{1/2}(\GB)$ 
\bel{affineD}
\begin{aligned}
KM(\hat\psi+\psi) &= KM(\hat\psi) + KM_0(\psi)\\
MK(\hat\phi+\phi) &= MK(\hat\phi) + MK_0(\phi).
\end{aligned}
\ee
\section{Landweber iteration (Neumann version)}
Define $N:H^{-1/2}(B)\ra H^{-1/2}(S)$ by
\bes
N(\psi) = \partial_n \calN(\psi) |_S.
\ee
We will show in this section that the previously described 
alternating technique is, in fact, the Landweber method applied to 
the operator equation 
\bel{operatorN}
N(\psi) = \tau.
\ee
Note that if $u$ is a solution of the Cauchy problem \eqref{cauchy},
then $N(\partial_n u|_B) = \tau$.  Moreover, if $\psi$ solves
equation \eqref{operatorN}, then $u=\calN(\psi)$ solves the Cauchy problem
\eqref{cauchy}.

Let $N_0$ be defined similarly to $N$ using $\calN_0$ in place of $\calN$ (i.e using homogeneous data). So
$N_0:H^{-1/2}(\GB)\ra H^{-1/2}(\GS)$ is a linear map and
\bes
N(\psi) = N(0) + N_0(\psi).
\ee
Hence the operator equation \eqref{operatorN} can be rewritten
\bel{operatorN0}
N_0(\psi) = \tau-N(0).
\ee

The Landweber method provides a regularization 
technique for solving equation \eqref{operatorN0} that proceeds by minimizing the functional
\bes
J(\psi) 
= \frac{1}{2} || \tau-N(\psi)||_{H^{-1/2}(\GS)}^2
= \frac{1}{2} || \tau-N(0) - N_0(\psi)||_{H^{-1/2}(\GS)}^2
\ee
using a steepest descent algorithm.  The gradient of $J$ at $\psi$ is
\bes
-N_0^*\left[\tau-N(0)-N_0(\psi)\right],
\ee
and we define
\bel{landweberN}
L_N(\psi) = \psi + a N_0^*\left[\tau-N(0)-N_0(\psi)\right]
\ee
where $a$ is a fixed constant chosen so that $0<a\le 1/||N_0||^2$.
The Landweber method then produces iterates $\psi_{k+1} = L_N(\psi_k)$ starting from  an initial value $\psi_0$, and the functions $\psi_k$ are then approximate solutions of the original operator equation \eqref{operatorN}.

Computation using the Landweber method requires knowledge of the adjoint $N_0^*$, which has a natural form given our chosen inner products.
\begin{lemma}\label{lem:Astar}
Let $\xi\in H^{-1/2}(\GS)$, and define $q$ to be the solution
of
\bes
\begin{aligned}
-\Lap q & = 0 &&\text{in $\Omega$}\\
\partial_n q & = \xi &&\text{on $\GS$}\\
q & = 0 &&\text{on $\GB$}.
\end{aligned}
\ee
Then $N_0^*(\xi) = \partial_n q|_B$.
\end{lemma}
\begin{proof}
Let $\psi\in H^{-1/2}(\GB)$ and $\xi\in H^{-1/2}(\GS)$ be arbitrary.
We then let $v$, $w$, $q$ and $r$ be the solutions of the following 
boundary value problems:
\bes
\begin{aligned}
-\Lap v &= 0 &\quad-\Lap w &= 0 &\quad-\Lap q &= 0 &\quad-\Lap r &= 0 &&\quad\text{in $\Omega$}\\
v &= 0  &\quad \partial_n w &= \partial_n v &\quad \partial_n q &= \xi &\quad r &=0 &&\quad \text{on $\GS$}\\
\partial_n v &= \psi  &\quad w &= 0 &\quad q  &= 0 &\quad  \partial_n r &=\partial_n q &&\quad \text{on $\GB$}.
\end{aligned}
\ee
Then
\ifjournal
\bel{AStar1}
	\begin{aligned}
\ip<N_0\psi,\xi>_{H^{-1/2}(\GS)} = \int_\Omega \nabla w\cdot \nabla q 
& = \int_{\pO} \partial_n w\; q \\
& = \int_S \partial_n w\; q \\
& = \int_S \partial_n v\; q 
 = \int_{\pO} \partial_n v\; q 
 = \int_{\Omega} \nabla v\cdot \nabla q.
\end{aligned}
\ee
\else
\bel{AStar1}
	\begin{aligned}
\ip<N_0\psi,\xi>_{H^{-1/2}(\GS)} = \int_\Omega \nabla w\cdot \nabla q 
 = \int_{\pO} \partial_n w\; q 
& = \int_S \partial_n w\; q \\
& = \int_S \partial_n v\; q 
 = \int_{\pO} \partial_n v\; q 
 = \int_{\Omega} \nabla v\cdot \nabla q.
\end{aligned}
\ee
\fi
The first equality follows from the definition of the inner product on $H^{-1/2}(\GS)$ and we have used $q=0$ on $B$ and $\partial_n v=\partial_n w$ on $S$ in the subsequent equalities.
Similarly,
\ifjournal
\bel{AStar2}
\begin{aligned}
\ip<\psi,\partial_n q|_B>_{H^{-1/2}(\GB)} = \int_\Omega \nabla v\cdot \nabla r 
 &= \int_{\pO} v\; \partial_n r \\
 &= \int_S v\; \partial_n r \\
 &= \int_S v\; \partial_n q 
 = \int_{\pO} v\; \partial_n q
 = \int_{\Omega} \nabla v\cdot \nabla q.
\end{aligned}
\ee
\else
\bel{AStar2}
\begin{aligned}
\ip<\psi,\partial_n q|_B>_{H^{-1/2}(\GB)} = \int_\Omega \nabla v\cdot \nabla r 
 = \int_{\pO} v\; \partial_n r 
 &= \int_S v\; \partial_n r \\ 
 &= \int_S v\; \partial_n q 
 = \int_{\pO} v\; \partial_n q 
 = \int_{\Omega} \nabla v\cdot \nabla q.
\end{aligned}
\ee
\fi
From equalities \eqref{AStar1} and \eqref{AStar2} we conclude
\begin{equation*}
\ip<N_0\psi,\xi>_{H^{-1/2}(\GS)} = \ip<\psi,\partial_n q|_B>_{H^{-1/2}(\GB)},
\end{equation*}
and since $\psi\in H^{-1/2}(B)$ is arbitrary, $N_0^*(\xi) = q|_B$ as claimed.
\end{proof}

The alternating method of \cite{KozlovMazya1990} is exactly the Landweber method
(with the constant $a=1$) applied to minimizing the functional $J$.
\begin{proposition}\label{prop:kmislandweberN}
For any $\psi\in H^{-1/2}(\GB)$, we have
\bes
L_N(\psi) = KM(\psi).
\ee
Consequently, the iterates produced by the (Neumann) Landweber method and the (Neumann) Kozlov-Maz'ya alternating method are identical.
\end{proposition}
\begin{proof}
Lemmas \ref{lem:linearsameN} and \ref{lem:affinesameN} proved below establish that
\bel{linearsame}
N_0^*N_0(\psi) = \psi-KM_0(\psi)
\ee
and
\bel{affinesame}
N_0^*(\tau-N(0)) = KM(0).
\ee
Hence
\begin{align*}
L_N(\psi) &= \psi + N_0^*(\tau-N(0)) - N^*_0 N_0(\psi) \\
 & = \psi +KM(0) - \left[\psi-KM_0(\psi)\right] \\
 & = KM(0) + KM_0(\psi)\\
 & = KM(\psi)
\end{align*}
where in the last step we have used the decomposition of the affine map $KM$.
\end{proof}
It remains to establish equations \eqref{linearsame} and \eqref{affinesame},
which is done in following two lemmas.
\begin{lemma}\label{lem:linearsameN}
For any $\psi$ in $H^{-1/2}(\GB)$,
\bes
N_0^*N_0(\psi) = \psi-KM_0(\psi).
\ee
\end{lemma}
\begin{proof}
Let $\psi\in H^{-1/2}(\GB)$ and let $v$ and $w$ solve the boundary-value
problems
\bes
\begin{aligned}
-\Lap v &= 0 &\quad-\Lap w &= 0 &&\quad\text{in $\Omega$}\\
v &= 0  &\quad \partial_n w &= 0 &&\quad \text{on $\GS$}\\
\partial_n v &= \psi  &\quad w &= v &&\quad \text{on $\GB$}.
\end{aligned}
\ee
Then $N_0(\psi)=\partial_n v|_{\GS}$ and $KM_0(\psi)=\partial_n w|_{\GB}$.  On the other hand, let
$q=v-w$.  Then
\bes
\begin{aligned}
-\Lap q & = 0 &&\text{in $\Omega$}\\
\partial_n q & = \partial_n v &&\text{on $\GS$}\\
q & = 0 &&\text{on $\GB$}
\end{aligned}
\ee
and hence 
$A^*(\partial_n v|_S)=\partial_n q|_B$.  But 
\bes
\partial_n q|_B = \partial_n v|_B- \partial_n w|_B = \psi -KM_0(\psi).
\ee
So
\bes
N_0^* N_0(\psi) = \psi - KM_0(\psi).
\ee
\end{proof}
\begin{lemma} \label{lem:affinesameN}
\bes
N_0^*(\tau-N(0)) = KM(0)
\ee
\end{lemma}
\begin{proof}
Let $v$  and $w$ be the solutions of
\bes
\begin{aligned}
-\Lap v & = f &\quad -\Lap w& = f &&\text{in $\Omega$}\\
v &= \sigma &\quad \partial_n w &= \tau &&\text{on $\GS$}\\
\partial_n v &= 0 &\quad w &=v &&\text{on $\GB$}.
\end{aligned}
\ee
Then $N(0) = \partial_n v|_\GS$ and $KM(0) = \partial_n w|_\GB$.  
Let $r = w-v$.  Then
\bes
\begin{aligned}
-\Lap r & = 0 &&\text{in $\Omega$}\\
\partial_n r & = \tau- \partial_n v|_S &&\text{on $\GS$}\\
r & = 0 &&\text{on $\GB$}.
\end{aligned}
\ee
By Lemma \ref{lem:Astar}, 
\bes
N_0^*(\tau-\partial_n v|_S) = \partial_n r|_B = \partial_n w|_B-\partial_n v|_B.
\ee
But $\partial_n v|_S=N(0) $,  $\partial_n w|_B=KM(0)$ and $\partial_n v|_B=0$.  
Hence $N_0^*(\tau-N(0)) = KM(0)$.
\end{proof}

\begin{remark} It follows from Lemma \ref{lem:linearsameN} that the operator $KM_0:H^{-1/2}(B)\ra H^{-1/2}(B)$ is self adjoint.  This fact was proved independently in \cite{KozlovMazya1990}, and it played a central role in their results. The new observation in the current work is that this self-adjointedness arises because of equation \eqref{linearsame}
and that the Kozlov-Maz'ya iteration is intimately connected with a minimization procedure.
\end{remark}
	
We have now proved all of the ingredients of Proposition \ref{prop:kmislandweberN}.  

To justify the use of Landweber iteration with relaxation factor $a=1$ in equation \eqref{landweberN}, we require an estimate of the norm of $N_0$.
\begin{lemma} The operator norm of $N_0$ satisfies $||N_0||\le 1$.
\end{lemma}
\begin{proof}
Let $v$, $w$, and $q$ be solutions of the following
boundary value problems:
\bes
\begin{aligned}
-\Lap v &= 0 &\quad-\Lap w &= 0 &\quad -\Lap q&=0 &&\quad\text{in $\Omega$}\\
v &= 0  &\quad \partial_n w &= 0 &\quad \partial_n q&=\partial_n v &&\quad \text{on $\GS$}\\
\partial_n v &= \psi  &\quad w &= v &\quad q&=0 &&\quad \text{on $\GB$}.
\end{aligned}
\ee
Then $N_0 \psi = \partial_n v|_{\GS} = \partial_n q|_{\GS}$ and
\bes
\begin{aligned}
    ||\psi||_{H^{-1/2}(\GB)} &= \int_{\Omega}\abs{\nabla v}^2\\
    ||N_0 \psi||_{H^{-1/2}(\GS)} &= \int_{\Omega}\abs{\nabla q}^2.
\end{aligned}
\ee
Now let $r=v-w$.  Then $r$ is harmonic, $r|_B=0$, and $\partial_n r = \partial_n v$ on $S$.  Hence $v-w=r=q$. Moreover,
\bes
\int_\Omega \nabla v\cdot \nabla w = \int_{\partial \Omega}  v\; (\partial_n w)\; =    
\int_{\partial \Omega}  w\; (\partial_n w) = \int_{\Omega} \abs{\nabla w} ^2
\ee
since $v = w$ on $\GB$ and $\partial_n w=0$ on $\GS$.  Hence
\bes
\begin{aligned}
||N_0 \psi||_{H^{-1/2}(\GS)}^2 = \int_{\Omega} |\nabla q|^2 = 
\int_\Omega \abs{\nabla (v-w)}^2 &= \int_\Omega \abs{\nabla v}^2 -2\nabla v\cdot\nabla w + \abs{\nabla w}^2\\
&= \int_\Omega \abs{\nabla v}^2 - \int_\Omega \abs{\nabla w}^2\\
&= ||\psi||_{H^{-1/2}(\GB)}^2 - \int_{\Omega} \abs{\nabla w}^2
\end{aligned}
\ee
So $||N_0 \phi||_{H^{1/2}(\GS)}\le ||\phi||_{H^{1/2}(\GB)}$ for all $h\in H^{1/2}(\GB)$ and hence $||N_0||\le 1$.
\end{proof}

It is well known that Landweber iteration, together with a stopping principle for the iterations, is
a regularization method for solving equation \eqref{operatorN} \cite{EnglHankeNeubauer}.  Translating
these standard results to Kozlov-Maz'ya iteration we obtain the following (which can be deduced, except 
for the statement concerning optimal convergence rates, from the original paper of \cite{KozlovMazya1990}).

\begin{proposition} \label{prop:landweberconvN}
Let $u$ be a solution of the Cauchy problem \eqref{cauchy} with Dirichlet data $\sigma$
and Neumann data $\tau$. Suppose $(\tau^\delta)$ are approximations of $\tau$ such that
$||\tau-\tau^\delta||_{H^{-1/2}(\GS)}<\delta$.
Let $\psi_0\in H^{-1/2}(\GB)$ and let $\psi_n^\delta$ be the first Kozlov-Maz'ya iterate
for the data $(\sigma,\tau_\delta)$ starting from $\psi_0$ such that
\bes
||\tau^\delta-N(\psi_n)||_{H^{1/2}(\GS)} < \lambda \delta
\ee 
where $\lambda>1$ is a fixed constant.  Then
\bes
\lim_{\delta\ra 0} ||\partial_n u|_\GB-\psi_n^\delta||_{H^{-1/2}(\GB)} = 0.
\ee
Moreover, the rate of convergence is order optimal.  That is, 
if $\partial_n u|_B = (N_0^* N_0)^\mu(\xi)$ for some $\mu>0$ and some $\xi\in H^{-1/2}(B)$, then
\bes
||\partial_n u|_B - \psi_n^\delta||_{H^{-1/2}(\GB)} \le c \delta^{2\mu/(2\mu+1)} E^{1/(2\mu+1)}
\ee
where $E=||\xi||_{H^{-1/2}(\GB)}$ and where $c$ is constant independent of the sequence $(\tau^\delta)$.
\end{proposition}

The previous result assumes that the Dirichlet data $\sigma$ is known exactly.  If $\sigma$ is only approximately known, this corresponds to error in the operator $N$. It is straightforward to
transform this error into increased uncertainty in the right-hand side of equation \eqref{operatorN0}.  To do this, we define $F_{DN}:H^{1/2}(\GS)\ra H^{-1/2}(\GS)$ by $F_{DN}(\gamma)=\partial_n z|_S$ where
\bes
    \begin{aligned}
        -\Delta z & =0 &&\text{in $\Omega$}\\
        z & = \gamma &&\text{on $\GS$}\\
        \partial_n z &= 0 &&\text{on $\GB$}.
    \end{aligned}
\ee
A simple computation (left to the reader) shows the following.
\begin{lemma} \label{lem:rewritesigmaerror}
Suppose that $\sigma_\epsilon\in H^{1/2}(\GS)$ satisfies
$||\sigma-\sigma_\epsilon||_{H^{1/2}(\GS)}<\epsilon$. Let $N^\epsilon$ be the corresponding operator in equation \eqref{operatorN}, and let 
$\tau^\epsilon = F_{DN}(\sigma^\epsilon-\sigma)$.  Then for all $\psi\in H^{-1/2}(\GB)$,
\bes
N(\psi)=N^\epsilon(\psi)+\tau^\epsilon.
\ee
As a consequence, the corresponding termination condition for Kozlov-Maz'ya iteration when there is error in both $\tau$ and $\sigma$ 
should be adjusted to
\bes
||\tau^\delta-\psi_n||_{H^{-1/2}(\GS)} < \lambda (\delta+||\tau^\epsilon||_{H^{-1/2}(\GS)}) 
\le \lambda(\delta+||F_{DN}||\epsilon). 
\ee
\end{lemma}
It is worth remarking that this is an inconvenient criterion to work with in practice: it requires both an estimate for
operator norm of the Dirichlet-to-Neumann map $F_{DN}$ as  well as the size of the error in $\tau$ in the space $H^{-1/2}(\GS)$, which has a rather abstract norm.  In fact, in many applications (including the work in \cite{MaxwellTrufferAvdonin} that motivates this paper) the Neumann data $\tau$ is known exactly (e.g. $\tau=0$) but there is error in the Dirichlet data $\sigma$.  Hence we now consider the Dirichlet version of operator equation \eqref{operatorN}.

\section{Landweber iteration (Dirichlet version)}\label{sec:landweberD}
In the previous section we recast Kozlov-Maz'ya iteration with operator $KM$ as a form of Landweber iteration by considering a map from Neumann data on $B$ to Neumann data on $S$ with fixed Dirichlet data on $S$.  The dual formulation
obtained by swapping the roles of Dirichlet and Neumann data corresponds to Kozlov-Maz'ya iteration with operator $MK$, but posing it
requires a little care.  The natural operator equation to consider is
\bel{operatorD}
D(\phi) = \sigma
\ee
where $D(\phi)=\calD(\phi)|_S$.  Defining $D_0$ using $\calD_0$ 
in place of $\calD$, we rewrite this equation as
\bes
D_0(\phi) = \sigma - D(0).
\ee
We would like to consider $D_0:H^{1/2}(B)\ra H^{1/2}(S)$,
but the challenge is to find inner products on these spaces such that
the resulting adjoint $D_0^*$ leads to a lemma analogous to Lemma \ref{lem:linearsameN}.  Unfortunately, this is not true for the norm defined in Section \ref{sec:H12}, and it is not clear how to adjust it to remedy this situation.

To circumvent these difficulties, pick any fixed $\hat\sigma\in H^{1/2}(\pO)$ such that $\hat\sigma|_S=\sigma$ and let $\hat\phi=\hat\sigma|_B$.  For example, one can obtain such a $\hat\phi$ by applying $MK$ to any element of $H^{1/2}(B)$.  Writing $\phi=\hat\phi+\eta$ for some $\eta\in H^{1/2}(B)$, equation \eqref{affineD} implies equation \eqref{operatorD} can be rewritten
\bel{operatorD00}
D_0(\eta) = \sigma-D(\hat\phi).
\ee
The gain here, as proved in the following lemma, is that the right-hand side of equation \eqref{operatorD00} belongs to $\Hoo(S)$, not just $H^{1/2}(S)$, and that if a solution exists, then $\eta\in\Hoo(B)$.
\begin{lemma}\strut\par

\begin{itemize}
\item If $\eta\in \Hoo(B)$ then $D_0(\eta)\in\Hoo(S)$.
\item If $\phi=\hat\phi+\eta$ solves equation \eqref{operatorD} then $\eta\in\Hoo(B)$.
\item The distribution $\sigma-D(\hat\phi)$ belongs to $\Hoo(S)$.
\end{itemize}
\end{lemma}
\begin{proof}
The claims all follow from the following observation: if $\phi_1$ and $\phi_2$
are distributions in $H^{1/2}(S)$ that admit extensions $\hat\phi_1$ and $\hat\phi_2$ in $H^{1/2}(\pO)$ that are equal on $B$, then $\phi_1-\phi_2\in H^{1/2}(S)$.  Indeed, $(\hat\phi_1-\hat\phi_2)|_B=0$ and 
$(\hat\phi_1-\hat\phi_2)|_S=\phi_1-\phi_2$, so $\phi_1-\phi_2\in \Hoo(S)$.
A similar result holds interchanging $S$ and $B$. 

Now suppose $\eta\in \Hoo(B)$. Since $0$ and $D_0(\eta)$ both admit
$H^{1/2}$ extensions to $\partial\Omega$ that are equal to $\eta$ on 
$B$, it follows that $D_0(\eta)=D_0(\eta)-0\in \Hoo(S)$.

Suppose $\phi=\hat\phi+\eta$ solves equation \eqref{operatorD}. Since $\phi$ and
$\hat\phi$ both admit $H^{1/2}(\pO)$ extensions 
that are equal to $\sigma$ on $S$, it follows that $\eta=\phi-\hat\phi\in \Hoo(B)$.

Finally, recall that $\hat\phi$ is the restriction of an $H^{1/2}(\partial\Omega)$ extension of $\sigma$ to $B$.
So $\sigma$ and $D(\hat\phi)$ admit $H^{1/2}(\pO)$ extensions that
are equal to $\hat\phi$ on $B$.  Hence $\sigma-D(\hat\phi)\in \Hoo(S)$.
\end{proof}

As a consequence of the previous lemma, we can interpret equation \eqref{operatorD00} as an operator equation from $\Hoo(B)$ to $\Hoo(S)$.
For $\eta\in\Hoo(B)$, let 
\bes
L_D(\eta) = \eta+D_0^*((\sigma-D(\hat\phi))-D_0(\eta)).
\ee
Here we treat $D_0$ as a map from $\Hoo(B)$ to $\Hoo(S)$ and
$D_0^*$ as a map from $\Hoo(S)$ to $\Hoo(B)$.  Landweber iteration
(with relaxation constant $a=1$) applied to equation \eqref{operatorD00} corresponds to starting with an
initial estimate $\eta_0\in \Hoo(B)$ and computing subsequent iterates
$\eta_{k+1}=L_D(\eta_k)$.  We then obtain approximate solutions
$\phi_k=\hat\phi+\eta_k$ of equation \eqref{operatorD}.
We will show that the iterates $\phi_k$ are exactly the iterates produces by  Kozlov-Maz'ya iteration with then operator $MK$ starting with the initial estimate 
$\phi_0=\hat\phi+\eta_0$.  To do this, we first compute the adjoint of $D_0$ and then prove analogues of Lemmas \ref{lem:linearsameN} and \ref{lem:affinesameN}.
\begin{lemma}\label{lem:Dstar}
Let $\gamma\in \Hoo(\GS)$, and let $q$ and $r$ be the solutions
of the boundary-value problems
\bel{Dstarp1}
\begin{aligned}
-\Lap q & = 0 &\quad -\Lap r &= 0 &&\quad\text{in $\Omega$}\\
q & = \gamma &\quad r &= 0 &&\quad\text{on $\GS$}\\
q & = 0 &\quad \partial_n r &= -\partial_n q &&\quad\text{on $\GB$}.
\end{aligned}
\ee
Then 
\bes
D_0^*(\gamma) = r|_{\GB}.
\ee.
\end{lemma}
\begin{proof}
Let $\phi\in \Hoo(B)$ and $\gamma\in \Hoo(S)$ be arbitrary, and let 
$v$ and $w$ be solutions of the following boundary-value problems:
\bes
\begin{aligned}
-\Lap v & = 0 &\quad -\Lap w &= 0 &&\text{in $\Omega$}\\
\partial_n v & = 0 &w &= 0 &&\text{on $\GS$}\\
v & = \phi & w &= \phi &&\text{on $\GB$}.
\end{aligned}
\ee
The equation for $w$ is well-posed since  $\phi\in \Hoo(S)$ and hence the prescribed boundary values lie in $H^{1/2}(\partial\Omega)$; a similar remark holds for $q$ in equation \eqref{Dstarp1}.

Notice that $v-w$ and $q$ are harmonic, equal zero on $B$, and  equal $D_0(\phi)$ and $\gamma$ respectively on $S$. By the definition of the inner-product on $\Hoo(S)$ we conclude that
\bes
\ip<D_0(\phi),\gamma>_{\Hoo(S)} = \int_\Omega \ip<\nabla(v-w),\nabla q>.
\ee
But
\bes
\int_{\Omega}\nabla v \cdot \nabla q = \int_{\partial\Omega} (\partial_n v)\; q = 0
\ee
since $q=0$ on $\GB$ and $\partial_n v = 0$ on $\GS$.  On the other hand,
\bes
-\int_{\Omega} \nabla w \cdot \nabla q = -\int_{\partial\Omega} w\; \partial_n q 
= \int_{\partial\Omega} w\; \partial_n r  
= \int_{\Omega}\nabla w \cdot \nabla r,
\ee
where we have used the fact that $w=0$ on $\GS$ and $\partial_n r=-\partial_n q$ on $\GB$.
Since $w$ and $r$ are harmonic, equal zero on $\GS$, and equal $\phi$ and $r|_{\GB}$ respectively on $\GB$ we have
\bes
\int_{\Omega}\nabla w \cdot \nabla r = \ip<\phi,r|_{\GB}>_{H^{1/2}_{00}(\GB)}.
\ee
Combining all of the equalities seen thus far we conclude
\bes
\ip<D_0(\phi),\gamma>_{H^{1/2}_{00}(\GS)} = \ip<\phi,r|_{\GB}>_{H^{1/2}_{00}(\GB)}
\ee
for all $\phi\in \Hoo(\GS)$.  Therefore $D_0^*(\gamma) = r|_{\GB}$.
\end{proof}

\begin{lemma}\label{lem:linearsameD}
	For any $\phi$ in $\Hoo(\GB)$,
	\bes
	D_0^*D_0(\phi) = \phi-MK_0(\phi).
	\ee
\end{lemma}
\begin{proof}
Let $\phi\in \Hoo(\GB)$ and let $u$, $v$, and $w$ be solutions of the following boundary-value problems:
\bes
\begin{aligned}
-\Lap u &= 0 &\quad-\Lap v &= 0 &\quad-\Lap w &= 0 &&\quad\text{in $\Omega$}\\
\partial_n u &= 0  & v &= 0 & w &= 0 &&\quad \text{on $\GS$}\\
u &= \phi  &\quad v &= \phi & \partial_n w  &= \partial_n u &&\quad \text{on $\GB$}.
\end{aligned}
\ee
Let $q=u-v$ and $r=v-w$.  Then $q$ and $r$ solve
\bes
\begin{aligned}
-\Lap q &= 0 &\quad-\Lap r &= 0 &&\quad\text{in $\Omega$}\\
 q &= u|_S  & r &= 0 &&\quad \text{on $\GS$}\\
 q &= 0  &\quad \partial_n r &= \partial_n v -\partial_n u
&&\quad \text{on $\GB$}.
\end{aligned}
\ee
Noting that $u|_S=D_0(\phi)$ and $(\partial_n v -\partial_n u)|_B=-\partial_n q|_B$ it follows from Lemma \ref{lem:Dstar} that 
\bes
D_0^*D_0(\phi)=r|_B=(v-w)|_B.
\ee
But $v|_B=\phi$ and $w|_B=MK_0(\phi)$.  Hence
\bes
D_0^*D_0(\phi) = \phi - MK_0(\phi).
\ee
\end{proof}

\begin{lemma} \label{lem:affinesameD}
\bes
D_0^*(\sigma-D(\tilde \phi)) = MK(\tilde\phi)-\tilde\phi.
\ee
\end{lemma}
\begin{proof}
Let $u$, $v$, and $w$ be solutions of the following boundary-value problems:
\bes
\begin{aligned}
-\Lap u &= f &\quad-\Lap v &= f &\quad-\Lap w &= f &&\quad\text{in $\Omega$}\\
u &= \sigma  & \partial_n v &= \tau & w &= \sigma &&\quad \text{on $\GS$}\\
u &= \tilde\phi  &\quad v &= \tilde\phi & \partial_n w  &= \partial_n v &&\quad \text{on $\GB$}.
\end{aligned}
\ee
Notice that $v|_S=D(\tilde\phi)$ and $w|_B=MK(\tilde\phi)$.  
Let $q=u-v$ and $r=w-u$. Then $q$ and $r$ solve
\bes
\begin{aligned}
-\Lap q &= 0 &\quad-\Lap r &= 0 &&\quad\text{in $\Omega$}\\
 q &= \sigma-D(\tilde\phi)  & r &= 0 &&\quad \text{on $\GS$}\\
 q &= 0  &\quad \partial_n r &= -\partial_n q
&&\quad \text{on $\GB$}.
\end{aligned}
\ee
Lemma \ref{lem:Dstar} implies $D_0^*(\sigma-D(\tilde\phi))=r|_B=(w-u)|_B$.
But $w|_B=MK(\tilde\phi)$ and $u|_B=\tilde\phi$.  Hence
\bes
D_0^*(\sigma-D(\tilde \phi)) = MK(\tilde\phi)-\tilde\phi.
\ee
\end{proof}
The proof of the following proposition exactly follows the proof of Proposition \ref{prop:kmislandweberN} using Lemmas \ref{lem:linearsameD} and \ref{lem:affinesameD} in place of  Lemmas \ref{lem:linearsameN} and \ref{lem:affinesameN}.  We omit the proof.
\begin{proposition}\label{prop:kmlandwebersameD}
For any $\eta\in \Hoo(\GB)$, we have
\bes
\hat\phi+L_D(\eta) = MK(\hat\phi+\eta).
\ee
Consequently, the iterates produced by the (Dirichlet) Landweber method and the (Dirichlet) Kozlov-Maz'ya alternating method are identical.
\end{proposition}

Just as with the Neumann formulation, the operator norm of $D_0$ 
is bounded above by $1$, which justifies setting the relaxation constant $a=1$ in our definition of $L_D$.
\begin{lemma} The operator norm of $D_0$ satisfies $||D_0||\le 1$.
\end{lemma}
\begin{proof}
Let $\eta\in \Hoo(B)$ and
let $u$, $v$, and $w$ satisfy the following boundary-value problems:
\bes
\begin{aligned}
-\Lap u &= 0 &\quad -\Lap v &= 0 &\quad -\Lap w &= 0 &&\text{in $\Omega$}\\
\partial_n u &= 0 &\quad v &= 0 &\quad w &= u  &&\text{on $\GS$} \\
u & = \eta &\qquad v & = \eta &\qquad w &= 0  &&\text{on $\GB$}.
\end{aligned}
\ee
Then $D_0 \eta = u|_{\GS}$ and
\bes
\begin{aligned}
    ||\eta||_{\Hoo(\GB)} &= \int_{\Omega}\abs{\nabla v}^2\\
    ||D_0 \eta||_{\Hoo(\GS)} &= \int_{\Omega}\abs{\nabla w}^2.
\end{aligned}
\ee
Notice that $u-v$ is harmonic, equals $0$ on $\GB$, and equals $u$ on $\GS$.  Hence $u-v=w$.  Moreover,
\bes
\int_\Omega \nabla u\cdot \nabla v = \int_{\partial \Omega}  (\partial_n u)\; v =    
\int_{B}  (\partial_n u) \; \eta = \int_{\partial \Omega}  (\partial_n u)\; u = \int_{\Omega} \abs{\nabla u} ^2
\ee
since $\partial_n u =0$ on $\GS$ and $v=u=\eta$ on $\GB$.  Hence
\ifjournal
\bes
\begin{aligned}
||D_0 \eta||_{\Hoo(\GS)}^2 &= \int_{\Omega} \nabla(u-v)\cdot\nabla (u-v)\\
& = \int_{\Omega} \abs{\nabla v}^2 - \int_{\Omega} \abs{\nabla u}^2
= ||\eta||_{\Hoo(\GS)}^2 - \int_{\Omega} \abs{\nabla u}^2
\end{aligned}
\ee
\else
\bes
||D_0 \eta||_{\Hoo(\GS)}^2 = \int_{\Omega} \nabla(u-v)\cdot\nabla (u-v) = \int_{\Omega} \abs{\nabla v}^2 - \int_{\Omega} \abs{\nabla u}^2
= ||\eta||_{\Hoo(\GS)}^2 - \int_{\Omega} \abs{\nabla u}^2
\ee
\fi
So $||D_0 \eta||_{\Hoo(\GS)}\le ||\eta||_{\Hoo(\GB)}$ for all $\eta\in \Hoo(\GB)$ and consequently $||D_0||\le 1$.
\end{proof}

Standard results for Landweber iteration (interpreted in the language of Kozlov-Maz'ya iteration) imply the following analogue of Proposition \ref{prop:landweberconvN}.
\begin{proposition} \label{prop:landweberconvD}
Let $u$ be a solution of the Cauchy problem \eqref{cauchy} with Dirichlet data $\sigma$
and Neumann data $\tau$. Suppose $(\sigma^\delta)$ are approximations of $\sigma$ such that
$||\sigma-\sigma^\delta||_{\Hoo(\GS)}<\delta$.
Let $\phi_0\in \hat\phi+\Hoo(\GB)$ (i.e. let $\phi_0$ admit an extension in $H^{1/2}(\Omega)$ that equals $\sigma$ on $S$) and let $\phi_n^\delta$ be the first Kozlov-Maz'ya iterate for the data $(\sigma_\delta,\tau)$ starting from $\phi_0$ such that
\bes
||\sigma^\delta-\phi_n||_{\Hoo(\GS)} < \lambda \delta
\ee 
where $\lambda>1$ is a fixed constant.  Then
\bes
\lim_{\delta\ra 0} ||u|_{\GB}-\phi_n^\delta||_{\Hoo(\GB)} = 0.
\ee

Moreover, the rate of convergence is order optimal.  That is, 
if $u|_B-\hat\phi = (D_0^* D_0)^\mu(\xi)$ for some $\mu>0$ and some $\xi\in \Hoo(B)$, then
\bes
||u|_B - \phi_n^\delta||_{H^{1/2}_{00}(\GB)} \le c \delta^{2\mu/(2\mu+1)} E^{1/(2\mu+1)}
\ee
where $E=||\xi||_{\Hoo(\GB)}$ and where $c$ is constant independent of the sequence $(\sigma^\delta)$.
\end{proposition}

The previous result assumes that $\tau$ is known exactly.  This actually holds in many application of interest where $\tau=0$ represents a stress-free or perfectly insulating boundary condition.  In particular, it holds in the motivating problem from in \cite{MaxwellTrufferAvdonin}.  If $\tau$ is only known approximately, then error in $\tau$ can be rewritten as expanded error in $\sigma$ in a procedure analogous to the one described in Lemma \ref{lem:rewritesigmaerror}.  

A more serious weakness of Proposition \ref{prop:landweberconvD} is that it assumes $||\sigma-\sigma_\delta||_{\Hoo(S)}\ra 0$ which morally implies that the values of $\sigma$ at the interface of $S$ and $B$ are known exactly.
We would prefer to have a theorem treating the case $||\sigma-\sigma_\delta||_{H^{1/2}(S)}\ra 0$.  Nevertheless, Proposition \ref{prop:landweberconvN} has some application in this case as well.  

Suppose $\sigma^\delta\ra\sigma$ in $H^{1/2}(S)$, and let $u^\delta$ and $v^\delta$ be solutions of the following 
boundary-values problems
\bes
\begin{aligned}
-\Lap u^\delta &= 0 &\quad-\Lap v^\delta &= 0 &&\quad\text{in $\Omega$}\\
u^\delta &= \sigma^\delta  &\quad \partial_n v^\delta &= \partial_n u^\delta &&\quad \text{on $\GS$}\\
\partial_n u^\delta &= 0  & v^\delta &= 0 &&\quad \text{on $\GB$}
\end{aligned}
\ee
and let $u^0$ and $v^0$ be the solution of these problems with $\sigma^\delta$
replaced with its true value $\sigma$.  
Then $u=u^\delta-v^\delta+w$ solves the Cauchy problem
\bes
\begin{aligned}
-\Lap u &= f &&\quad\text{in $\Omega$}\\
\partial_n u &= \tau&&\quad \text{on $\GS$}\\
u &= \sigma^\delta&&\quad \text{on $\GS$}.
\end{aligned}
\ee
if and only if $w$ solves the Cauchy problem
\bel{cauchy-reduced}
\begin{aligned}
-\Lap w &= f &&\quad\text{in $\Omega$}\\
\partial_n w &= \tau&&\quad \text{on $\GS$}\\
w &= v^\delta&&\quad \text{on $\GS$}.
\end{aligned}
\ee
Noting that $v^\delta\in\Hoo(S)$, and that $v^\delta|_S\ra v^0|_S$ in $H^{1/2}(S)$, Proposition \ref{prop:landweberconvN} can 
be applied to the Cauchy problems \eqref{cauchy-reduced}.  The stopping criterion then involves the operator norm of the map taking $\sigma^\delta\in H^{1/2}(S)$ to $v^\delta|_S\in \Hoo(S)$.

\section{Conjugate gradient alternative}\label{sec:cg}

Since the Kozlov-Maz'ya alternating method is simply a form of 
the Landweber method, it becomes clear how it might be effectively accelerated. One standard, attractive choice is to use the conjugate gradient method. This strategy, together
with the Morozov discrepancy principle, provides a fast, order-optimal 
regularization scheme (see, eg., \cite{HankeCG}).

For definiteness we treat the Dirichlet case and consider the normal equation
\bes
D_0^*D_0(\eta) = D_0^*(\sigma-D(\hat\phi)).
\ee
The conjugate gradient algorithm for this problem then reads

\quad\LinesNumbered
\begin{algorithm}[H]
$r_0 = \sigma-D(\hat\phi+\eta_0)$\;
$q_0=D_0^* r_0$\;
$d = q_0$\;
$k = 0$\;
\While{true}
{
  $s=D_0 d$\;
  $\alpha = || q_k ||^2_{\Hoo(\GB)} /||s||_{\Hoo(\GS)}$\;
  $\eta_{k+1} = \eta_k+\alpha d$\;
  $r_{k+1} = r_k - \alpha s$\;
  $q_{k+1} = D_0^*r_{k+1}$\;
  $\beta = || q_{k+1} ||^2_{\Hoo(\GB)} / || q_k ||^2_{\Hoo(\GB)}$\;
  $d = q_{k+1}+\beta d$\;
  $k=k+1$\;
}
\caption{Conjugate gradient version of Dirichlet Landweber approach}
\label{alg1}
\end{algorithm}

Using the discrepancy principle, the main loop is terminated when $||r_k||_{\Hoo(S)}$ is sufficiently small, and the regularized solution of
the Cauchy problem is then $\calD(\hat\phi+\eta_k)$. 

The computation of $\alpha$ and $\beta$
requires computation of several norms in $\Hoo(\GB)$ and $\Hoo(\GS)$, each of which would appear to require the solution of a boundary-value problem. We circumvent this difficulty by representing each element $\gamma$ of $H^{1/2}_{00}(\GB)$ (i.e. $q_k$ and $d$ in the algorithm)
by a harmonic function that is equal to 
$\gamma$ on $\GB$ and  equal to $0$ on $\GS$; the norm
is then easy to compute according to the definition in Section \ref{sec:H12}.
A similar principle applies
to the variables $r_k$ and $s$ in $\Hoo(\GS)$ which are represented by harmonic functions that equal zero on $\GB$.  For this convention to be effective, we need to be able to compute the action of $D_0$ and $D_0^*$, which the following lemmas show is remarkably easy.
\begin{lemma}
Suppose $w$ is harmonic and equals zero on $S$.  Then
$$
z=\calD_0(w|_B)-w
$$
is harmonic, equals $D_0(w|_B)$ on $S$ and equals zero on $B$.
\end{lemma}
\begin{proof}
Since $z$ is a difference of harmonic functions it is harmonic.
Since $w|_S=0$ we have
\bes
z|_S = \calD_0(w|_B)|_S = D_0(w|_B)
\ee
by definition of $D_0$ in terms of $\calD_0$.  On the other hand, $\calD_0(w|_B)|_B=w|_B$, by the definition of $\calD_0$.
So $z|_B=0$.
\end{proof}
\begin{lemma}
Suppose $w$ is harmonic and equals zero on $B$.  Then
\bel{harmonicN0}
z=-\calN_0(\partial_n w|_S)
\ee
is harmonic, equals $D_0^*(w|_S)$ on $B$ and equals zero on $S$.
\end{lemma}
\begin{proof}
That $z$ is harmonic and equals zero on $S$ follows immediately
from the definition of $\calN_0$.  On the other hand,
inspecting Lemma \ref{lem:Dstar} with $w$ playing the role of $q$ and
$z$ playing the role of $r$ in equations \eqref{Dstarp1}
we see that $z|_B = D_0^*(w|_S)$.
\end{proof}
It is worth remarking that $z$ in equation \eqref{harmonicN0}
satisfies the weak formulation that $z|_S=0$ and
$$
\int_\Omega \nabla z \cdot \nabla \chi = \int_\Omega \nabla w \cdot \nabla \chi
$$
for all test functions $\chi$ that equal zero on $S$. Hence the exterior derivative $\partial_n w$ need not be explicitly found 
when solving for $z$.

Starting with an initial value $\hat\phi+\eta_0\in\hat\phi+\Hoo(B)$, 
let $u_0$ be the solution of
\bes
\begin{aligned}
-\Laplacian u_0 &= f &&\text{in $\Omega$}\\
u_0 &= \sigma &&\text{on $\GS$}\\
u_0 &= \hat\phi+\eta_0&&\text{on $\GB$}.\\
\end{aligned}
\ee
Then Algorithm \ref{alg1}, can be rewritten as follows.

\quad
\begin{algorithm}[H]
$r_0 = u_0-\calD(u_0|_B)$\;
$q_0=-\calN_0(\partial_n r_0|_B)$ \;
$d = q_0$\;
$k = 0$\;
\While{true}
{
  $s= \calD_0(d|_B) - d$\;
  $\alpha = \int_\Omega \abs{\nabla q_k}^2 \Big / \int_\Omega\abs{\nabla s}^2$\;
  $u_{k+1} = u_k+\alpha d$\;
  $r_{k+1} = r_k - \alpha s$\;
  $q_{k+1} = -\calN_0(\partial_n r_{k+1}|_S)$\;
  $\beta = \int_\Omega \abs{\nabla q_{k+1}}^2 \Big / \int_\Omega \abs{\nabla q_{k}}^2$\;
  $d = q_{k+1}+\beta d$\;
  $k=k+1$\;
}
\caption{Simplified Dirichlet conjugate gradient approach}
\label{alg2}
\end{algorithm}

When the main loop is terminated (e.g. using the discrepancy principle), 
the regularized solution of the Cauchy problem is then $\calD(u_k|_B)$. 
Each iteration of the loop requires solving exactly two boundary-value problems (one for $\calD_0$ and one for $\calN_0$), 
just as for Kozlov-Maz'ya iteration.  In Section \ref{sec:numerics} we demonstrate how the number of iterations needed for the conjugate gradient algorithm can be substantially less than standard Kozlov-Maz'ya iteration.

\section{Variations of the conjugate gradient approach}
We have worked with solving the equation
\bes
D(\phi) = \sigma
\ee
where $D:H^{1/2}(\GB)\ra H^{1/2}(\GS)$.
By changing the source or target spaces to be $L^2$
spaces, one obtains three alternative possibilities for the conjugate gradient algorithm.  

\begin{itemize}
\item{\bf[$L^2(\GB)\ra L^2(\GS)$]}\par
This variation was treated in \cite{HaoLesnic2000}.  Because the choice of domain has lower regularity than $H^{1/2}$, one
expects the reconstructed solutions to exhibit lower regularity than Kozlov-Maz'ya iteration.  Indeed, one step of the algorithm
involves a boundary condition of the form
\bel{reg-loss}
u|_{\GB} = \partial_n v|_{\GB},
\ee
where $v$ is a previously computed harmonic function.  This step has the effect of lowering the regularity of $u$ and is perhaps
responsible for oscillations observed in \cite{HaoLesnic2000} Figures 2, 4, and 6.  In particular, we performed a reconstruction of \cite{HaoLesnic2000} Figure 2 using the $H^{1/2}\rightarrow H^{1/2}$ method (as well as the $H^{1/2}\rightarrow L^2$ method described below) and these oscillations are absent.
	
\item{\bf[$L^2(\GB)\ra H^{1/2}(\GS)$]}
This approach appears in \cite{Knowles2004}, although it is not presented as such.
That paper considers the functional
\bes
G(\phi) = \int_\Omega \abs{\nabla(v-w)}^2
\ee
where $v$ is the solution of 
\bes
\begin{aligned}
-\Laplacian	v &= f &&\text{in $\Omega$}\\
v &= \sigma &&\text{on $\GS$}\\
v &= \phi &&\text{on $\GB$}.
\end{aligned}
\ee
and $w$ is the solution of
\bes
\begin{aligned}
-\Laplacian	w &= f &&\text{in $\Omega$}\\
\partial_n w &= \tau &&\text{on $\GS$}\\
w &= \phi &&\text{on $\GB$}.
\end{aligned}
\ee
Noting that $(v-w)$ is harmonic and equal to zero on $\GB$, we see that $G$ can be rewritten as
\bes
G(\phi) = || (v-w)|_{\GS} ||_{\Hoo(\GS)}^2 = || \sigma-w|_{\GS} ||_{\Hoo(\GS)}^2,
\ee
which is the functional being minimized by the Dirichlet Landweber procedure
presented in Section \ref{sec:landweberD}.  
However, \cite{Knowles2004} initially uses the $L^2(\GB)$ gradient of $G$.  Since $G$ is not defined on all of $L^2$
(we do not expect solutions with $L^2$ boundary data to lie in $H^1(\Omega)$), the $L^2$ gradient leads
to a loss of regularity, and the algorithm has a step similar to equation \eqref{reg-loss}.
This trouble is ameliorated in \cite{Knowles2004}  by introducing a smoothing step, effectively recasting the 
domain as a subspace of $H^1(\GB)$ and factoring the map through $L^2(\GB)$.  
Since the norm used for $H^1(B)$ involves a PDE defined only on the domain boundary,
the additional step adds some complication to the algorithm when compared to Kozlov-Maz'ya iteration.

\item{\bf[$H^{1/2}(\GB)\ra L^2(\GS)$]}\par
This combination does not appear to have been previously addressed in the literature, and has some potential interest.  Since the domain is $H^{1/2}(\GB)$, we
expect a reconstruction with higher regularity than the method of \cite{HaoLesnic2000}.  And since the range is $L^2(\GS)$, 
the associated stopping principle will involve $L^2$ error estimates, which are much easier to obtain than the rather abstract $H^{1/2}$ error estimates.

As in Section \ref{sec:landweberD}, we consider the operator equation
\bel{operatorDagain}
(\iota\circ D)(\eta) = \sigma
\ee
where $D:H^{1/2}(B)\ra H^{1/2}(S)$ and where $\iota:H^{1/2}(S)\ra L^2(S)$
is the natural embedding. Note that we work with $H^{1/2}(B)$ rather
than the more awkward space $\Hoo(B)$.
Equation \eqref{operatorDagain} can be rewritten
\bes
(\iota\circ D_0)(\eta) = \sigma-\iota(D_0)(0)),
\ee
and to apply the Landweber or conjugate gradient methods we need to
be able to compute $(\iota\circ D_0)^*$.
\begin{proposition}
Let $\gamma\in L^2(S)$, and let $w$ be the solution of 
\bel{iDstar}
\begin{aligned}
-\Lap w &= 0 &&\text{in $\Omega$}\\
\partial_n w &= \gamma &&\text{on $S$}\\
\partial_n w &= -\frac{1}{|B|}\int_S \gamma &&\text{on $B$}\\
\frac{1}{|B|}\int_B w &=  \int_S\gamma.
\end{aligned}
\ee	
Then $(\iota\circ D_0)^*(\gamma)=w|_B$.
\end{proposition}
\begin{proof}
We first note that there is a solution $w\in H^{1}(\Omega)$ of system
\eqref{iDstar}. Indeed, the Neumann data (which belongs to $L^2(\pO)\subseteq H^{-1/2}(\pO)$) satisfies the compatibility condition $\int_{\pO} \partial_n w = 0$.  Hence the PDE admits a solution in $H^{1}(\Omega)$ determined
uniquely up to a constant.  The final equation then determines the value of the constant.

Now let $\phi\in H^{1/2}(B)$ and  $\gamma\in L^2(S)$ be arbitrary.
Let $w$ be the solution of system \eqref{iDstar} and let $v$ and $q$ solve
\bes
\begin{aligned}
-\Laplacian	v &= 0  &\quad -\Lap q &= 0     &&\text{in $\Omega$}\\
\partial_n v &=  &\quad \partial_n q &= 0 &&\text{on $\GS$}\\
 v &= \phi &\quad q &= w|_B &&\text{on $\GB$}.\\
\end{aligned}
\ee
Then $D_0(\phi)=v|_S$ and
\bes
\begin{aligned}
\ip<(\iota\circ D_0)(\phi),\gamma>_{L^2(S)} &=\int_S v \gamma \\
&=  \int_{\pO} v\, \partial_n w +\int_B\left[ v\; \frac{1}{|B|}\int_S \gamma\right]\\
&=  \int_{\Omega} \nabla v \cdot \nabla w +\left[\frac{1}{|B|}\int_{B} v\right]  \left[\int_S \gamma\right].
\end{aligned}
\ee
Now 
\bes
\int_S\gamma = \frac{1}{|B|}\int_B w = \frac{1}{|B|}\int_B q.
\ee
Moreover,
\bes
\int_\Omega \nabla u\cdot\nabla w = \int_{\partial\Omega}\partial_n u\; w
= \int_{\partial\Omega} \partial_n u\; q = \int_{\Omega}\nabla u\cdot\nabla w
\ee
since $\partial_n u=0$ on $S$ and $q=w$ on $B$.  Combining 
all these equations we conclude
\bes
\begin{aligned}
\ip<(\iota\circ D_0)(\phi),\gamma>_{L^2(S)} &= \int_\Omega \nabla v\cdot \nabla q + 
\left[\frac{1}{|B|}\int_{B} v\right] \left[\frac{1}{|B|}\int_{B} q\right]\\ 
&= \ip<v|_B,q|_B>_{H^{1/2}(B)}\\
&= \ip<\phi,q|_B>_{H^{1/2}(B)}.
\end{aligned}
\ee
Hence $(\iota\circ D_0)^*(\gamma) = q|_B$.
\end{proof}

The conjugate gradient algorithm for this problem starts with
an initial value $\phi_0\in H^{1/2}(B)$ and a corresponding
$u_0=\calD(\phi_0)$.  Given $\gamma\in L^2(S)$, we define 
$W_0(\gamma)=w|_B$ where $w$ is  the solution of system \eqref{iDstar}.
We then obtain an analogue of Algorithm \ref{alg2} by 
tracking elements of $H^{1/2}(B)$ as harmonic functions with zero Neumann data on $S$.

\begin{algorithm}[H]
$r_0 = \sigma - u_0|_S$\;
$q_0= \calD_0(W_0(r_0))$ \;
$d = q_0$\;
$k = 0$\;
\While{true}
{
  $\alpha = \int_\Omega \abs{\nabla q_k}^2 \Big / \int_S\abs{d|_S}^2$\;
  $u_{k+1} = u_k+\alpha d$\;
  $r_{k+1} = r_k - \alpha d|_S$\;
  $q_{k+1} = \calD_0(W_0(r_k))$\;
  $\beta = \int_\Omega \abs{\nabla q_{k+1}}^2 \Big / \int_\Omega \abs{\nabla q_{k}}^2$\;
  $d = q_{k+1}+\beta d$\;
  $k=k+1$\;
}
\caption{$H^{1/2}(B)\rightarrow L^2(S)$ conjugate gradient approach}
\label{alg3}
\end{algorithm}

Upon exit, the regularized solution of the Cauchy problem is
simply $u_k$.

\end{itemize}

\section{Numerical Results}\label{sec:numerics}

Let $\Omega$ be the domain in the plane bounded above by the $x$-axis for $-1\le x \le 1$ and bounded below by a parabola passing through the points $(-1,0)$, $(1,0)$, and $(0,-d)$ (Figure \ref{fig:domain}, left). We take $S$ to be the region on the $x$-axis with $-1<x<1$ and $B$ to be the portion of the boundary with $y<0$.  The Cauchy problem to solve is
\bes
\begin{aligned}
	-\Lap u &= f_0 &&\text{in $\Omega$}\\
	u &= \sigma && \text{on $S$}\\
	\partial_n u &= 0 && \text{on $S$}
\end{aligned}
\ee
where $f_0$ is a constant. This is a model for a glaciological inverse problem where $\Omega$ is the cross-section of a glacier and $u$ represents the component of ice velocity orthogonal to the cross-section.  The homogeneous Neumann condition arises as a consequence of a zero-stress hypothesis at the ice surface $S$, and surface velocity measurements are represented by $\sigma$.

\begin{figure}
	\centering
\includegraphics{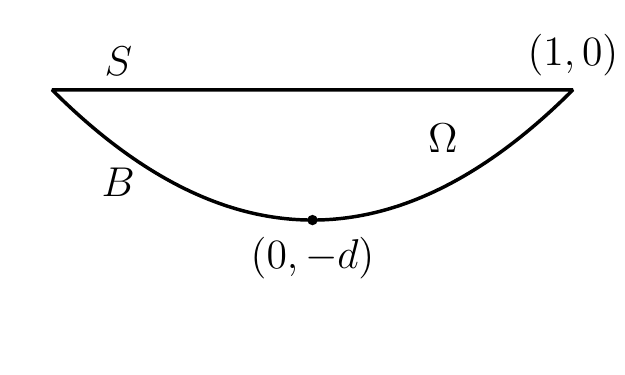}\includegraphics{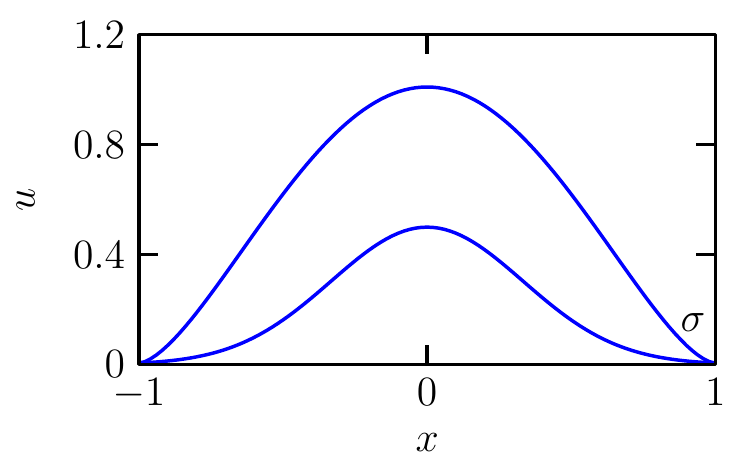}
\caption{Left: computational domain.  Right: True values of $\sigma=u|_S$  and $\phi=u|_B$.}
\label{fig:domain}
\end{figure}
\begin{figure}
	\centering
\includegraphics{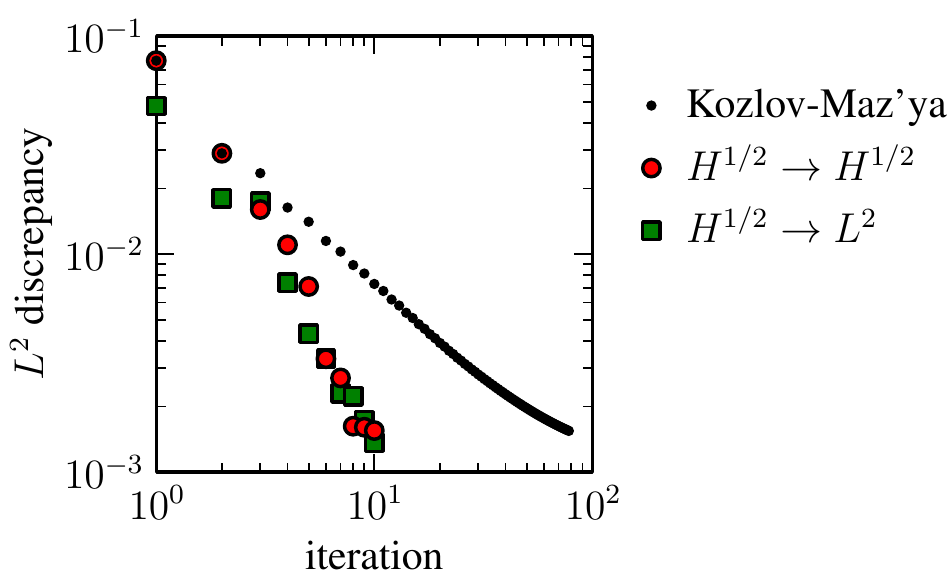}
\caption{History of the $L^2$ discrepancies for the $p=0.1$ reconstruction.
}
\label{fig:disc_history}
\end{figure}
We consider synthetic data $\sigma$ obtained by numerically solving the problem
\bel{uforward}
\begin{aligned}
	-\Lap u &= f_0 &&\text{in $\Omega$}\\
	\partial_n u &= 0 && \text{on $S$}\\
	u &= \phi && \text{on $B$}
\end{aligned}
\ee
where $\phi$ is a prescribed function; then $\sigma=u|_S$.  We used
a bump-function 
\bes
\phi(x,y)= u_0 \exp\left(-\frac{1}{2}\,{x^2}/{s^2}\right),
\ee
where $u_0$ and $s$ are constants.  Figure \ref{fig:domain} (right)
illustrates values of $u$ on $B$ and $S$.
\begin{figure}
		\centering
\includegraphics{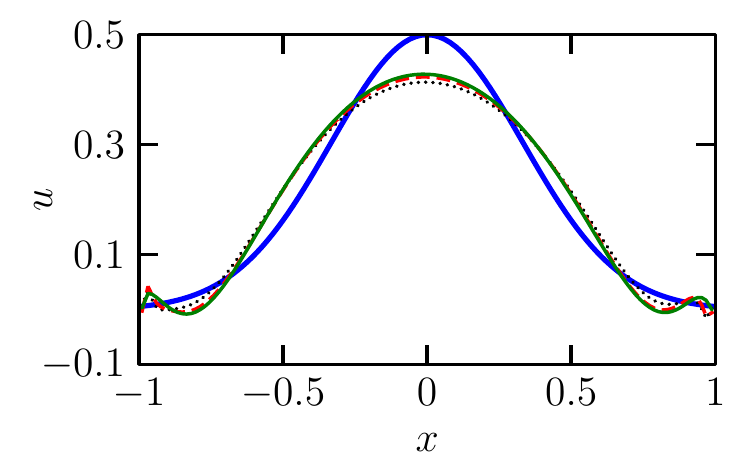}\ifjournal\quad\fi\includegraphics{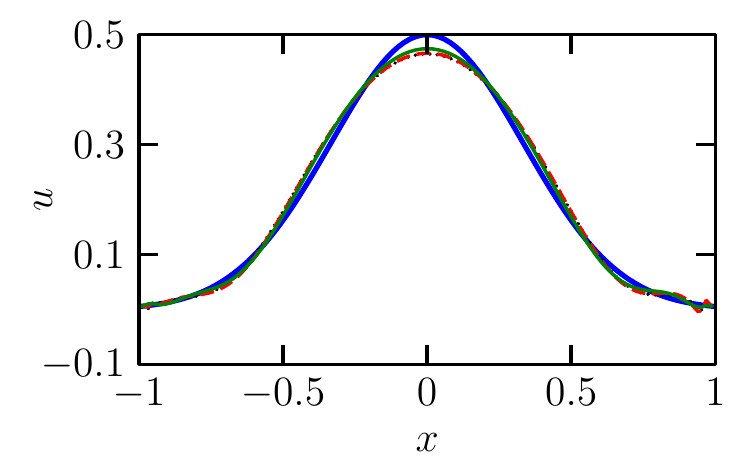}
\caption{Reconstruction of $u$ on $B$.  \ifjournal Above: \else Left: \fi $p=1$. \ifjournal Below: \else Right: \fi $p=0.1$.
Kozlov-Mazya (black, dotted), $H^{1/2}\rightarrow H^{1/2}$ (red, dashed), $H^{1/2}\rightarrow L^2$ (green, solid), true value (blue, solid).}
\label{fig:basal_match}
\end{figure}
\begin{figure}
	\centering
\includegraphics{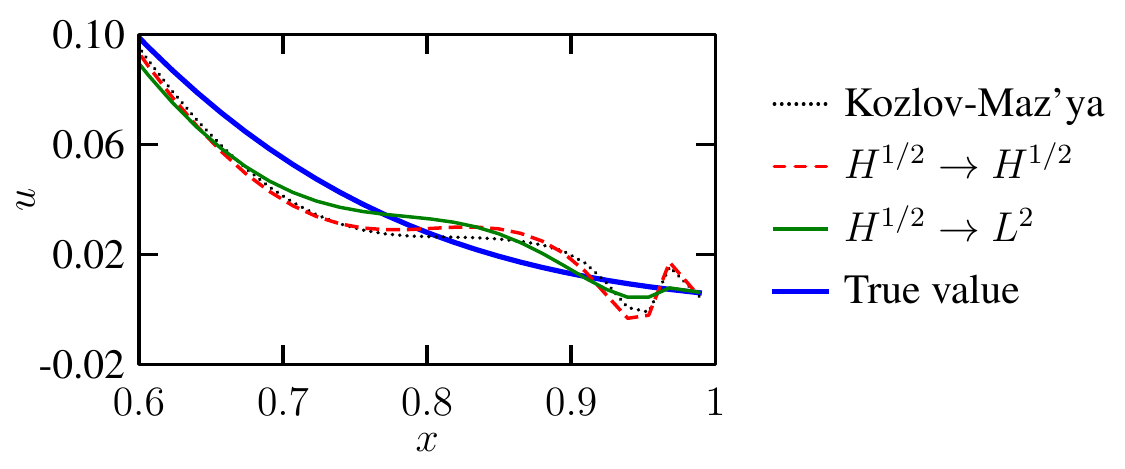}
\caption{Detail of Figure \ref{fig:basal_match}, $p=0.1$, near the interface of $S$ and $B$.}
\label{fig:basal_match_zoom}
\end{figure}

All computations were done using the finite element method (and in particular using the FEniCS \cite{FEniCS} framework).
In the following we used depth $d=1/2$, forcing term $f_0=8$, peak basal speed $u_0=1/2$,  and standard deviation $s=1/3$.  The forcing term was selected so that the peak value $u_{\rm max}$ of the solution $u$ of system \eqref{uforward}  was approximately $1$.

The surface measurements $\sigma$ were perturbed by spatially uncorrelated gaussian noise with standard deviation $u_{\rm max} \frac{p}{100}$, where 
$p$ is a constant describing `percent-noise'.
We then applied standard Kozlov-Maz'ya iteration, the $H^{1/2}\rightarrow H^{1/2}$ conjugate gradient method (Algorithm \ref{alg2}), and the $H^{1/2}\rightarrow L^2$
conjugate gradient method (Algorithm \ref{alg3}) to solve the inverse problem.
The starting estimate was $\phi_0=0$ for all reconstructions.  

Estimating error in the $H^{1/2}(S)$ norm for use in the discrepancy principle 
is impractical, whereas we have good estimates for the error in the $L^2(S)$ norm.  For all three algorithms, we therefore terminated early based on the $L^2$ discrepancy.  This is formally justified only for the $H^{1/2}\rightarrow L^2$ algorithm. Nevertheless, the other two algorithms using this modified discrepancy principle appeared to remain stable.  It would be interesting to have a formal proof of this observation.

Figure \ref{fig:disc_history} shows a graph of the $L^2$ discrepancies
\bes
||u-u_k||_{L^2(S)}
\ee
for the $p=0.1$ reconstruction for each of the three algorithms.
As would be expected, the conjugate gradient based algorithms required 
substantially fewer iterations to reach their target discrepancies.
This improvement was less substantial for the larger (and more physically relevant) value of $p=1$  where the conjugate gradient algorithms each used six iterations and Kozlov-Maz'ya iteration used nine iterations.

All three algorithms gave qualitatively similar reconstructions for $p=1$ and $p=0.1$ (Figure \ref{fig:basal_match}).  The algorithms using $H^{1/2}(S)$ for the surface norm tended to have stronger oscillations near the interface of $S$ and $B$.  Figure \ref{fig:basal_match_zoom} shows a detail of these oscillations near the boundary point $(1,0)$ in the $p=0.1$ case.  Although present in all three reconstructions, the oscillations were more damped by the  $H^{1/2}\rightarrow L^2$ algorithm.

Of the three algorithms used, the $H^{1/2}\rightarrow L^2$ algorithm had a number of slight advantages.  In addition to having the speed of the conjugate gradient algorithm, it has a provably rigorous stopping principle for an easily obtained error estimate, and it generated subtly better reconstructions near the interface of $S$ and $B$.

\section*{Appendix}
We recall here standard facts about Sobolov spaces.  Unless otherwise noted,
we use the definitions and notation of the careful exposition in \cite{McleanElliptic}.

Let $\Omega\subseteq\Reals^n$ be a connected, relatively compact open set with a Lipschitz boundary  $\partial \Omega$. Suppose
$S$ and $B$ are open subsets of $\partial \Omega$ sharing a common
boundary $\Pi$ and that $\pO=S\cup B\cup \Pi$ is a Lipschitz dissection as defined in \cite{McleanElliptic}.

For $s\in[-1,1]$, we define $H^{s}(\partial \Omega)$ as in \cite{McleanElliptic}.
In particular, the trace map from $H^1(\Omega)$ to $H^{1/2}(\pO)$ is a continuous surjection that admits a continuous right inverse.

The space $H^{s}(S)$ is then defined as restrictions to $S$
of distributions in $H^{s}(\pO)$, and is given the quotient norm,
so
\bes
||u||_{H^{s}(S)} = \inf\{||U||_{H^s(\pO)}:U|_{S}=u\}.
\ee
We define $\Hoos(S)$ to be the closure in $H^{s}(\pO)$ of the Lipschitz functions  with compact support in $S$ . In particular,
$\Hoos(S)$ is a closed subspace of $H^s(\pO)$, though we will commonly identify such functions as elements of $H^s(S)$.  Because
of the regularity of the common boundary $\Pi$, 
\bes
\Hoos(S) = \{ u\in H^s(S): u|_{B}=0\}.
\ee
Spaces of distributions on $B$ are defined similarly.

The bilinear form on $\Lip(\pO)\times \Lip(\pO)$
given by
\bes
\left<u,v\right>=\int_{\pO} uv\;
\ee
extends continuously to a bilinear form on $H^{-s}(\pO)\times H^{s}(\pO)$, and the map $u\mapsto \left<u,\cdot\right>$ gives an
isomorphism between $H^{-s}(\pO)$ and $H^{s}(\pO)^*$.  Again,
because of the regularity of the common boundary $\Pi$, there is
an isomorphism $\iota$ between $\Hoons(S)$ and $(H^{s}(S))^*$ defined as follows: for $u\in \Hoons(S)$ and $v\in H^{s}(S)$, 
\bes
(\iota u)(v) = \left<u,V\right>
\ee
where $V\in H^s(\pO$ is any function with $V|_{S}=v$. By reflexivity we also have an isomorphsim between $H^s(S)$
and $\Hoons(S)^*$.  Note that \cite{McleanElliptic} denotes 
$\Hoos$ by $\tilde H^s$.

\subsection*{Mixed Boundary Value Problems}
We wish to solve
\bel{poisson-mixed}
\begin{aligned}
-\Laplacian u &= f &&\text{in $\Omega$}\\
u &= \phi &&\text{on $S$}\\
\partial_n u &= \tau &&\text{on $B$}.
\end{aligned}
\ee
Let $H^1(\Omega,S)=\{u\in H^1(\Omega):u|_{S}=0\}$. This
is a closed subspace of $H^1(\Omega)$ (being the kernel of restriction to the boundary followed by restriction to $S$).  For $u,v\in H^1(\Omega,S)$ let
\bes
A(u,v) = \int_\Omega \nabla u\cdot \nabla v.
\ee
This is evidently a continuous bilinear form.  Moreover, it is strongly coercive
so long as $S$ is nonempty (since we have assumed additionally that $S$ is open). The argument is completely analogous to the corresponding 
one for the well-known case where $S=\pO$.

Now suppose $f\in (H^1(\Omega))^*$, $\phi\in H^{1/2}(S)$, and $\tau\in H^{-1/2}(B)$.
A weak solution of equation \eqref{poisson-mixed} is a function $u\in H^1(\Omega)$ such that $u=\phi$ in $S$ and such that
\bel{weak-form}
A(u,v) = \int_\Omega f v + \int_{B} \tau\, v 
\ee
for all $v\in H^1(\Omega,\Gamma_1)$. The integrals on the right-hand side of equation \eqref{weak-form} are to be interpreted as shorthand for the application of linear functionals.  In particular, the integral over $B$ makes
sense for if $v\in H^{1}(\Omega,S)$ then $v|_{S}=0$ and
hence $v|_{\pO}\in H^{1/2}_{00}(B) = H^{-1/2}(B)^*$.

Pick $w\in H^{1}(\Omega)$ such that $w|_{S}=\phi$; indeed since the trace and restriction operators have continuous right inverses, we can pick a $w$ such that its norm in $H^{1}(\Omega)$ is controlled by the norm of $\phi$ in $H^{1/2}(S)$.  Then $u=w+h$ solves equation \eqref{weak-form} if and only if $h\in H^{1}(\Omega,S)$
and 
\bel{weak-form-nonhom}
\int_\Omega \nabla h \cdot \nabla v = \int_\Omega (f v -\nabla w\cdot\nabla v) + \int_{B} \tau\, v 
\ee
for all $v\in H^1(\Omega,\Gamma_1)$. The right-hand side of equation  \eqref{weak-form-nonhom} is a continuous linear functional on $v\in H^1(\Omega,\Gamma_1)$.  So the the Lax-Milgram theorem implies that there is a unique solution $h$ of equation \eqref{weak-form-nonhom} and hence a unique solution 
$u$ of equation \eqref{weak-form}.  Moreover, the Lax-Milgram theorem (along with the aforementioned control on the size of $||w||_{H^1(\Omega)}$) implies there are constants $c_1$ and $c_2$ such that
\bel{mixed-cont}
\begin{aligned}
||u||_{H^1(\Omega)} &\le c_1 (||w||_{H^1(\Omega)} +||h||_{H^1(\Omega,S)})\\
& \le c_2 (||f||_{(H^{1}(\Omega,S))^*} + ||\phi||_{H^{1/2}(S)} + ||\tau||_{H^{-1/2}(B)}).
\end{aligned}
\ee

\subsection*{Boundary Neumann Data}
Suppose $u\in H^1(\Omega)$ and $-\Lap u = f$ where $f\in (H^{1}(\Omega))^*$.
Then we define $\partial_n u|_{\pO}\in H^{-1/2}(\pO)$ by 
\bel{neumann-def}
\int_{\pO} (\partial_n u) \phi = \int_{\Omega} \nabla u\cdot\nabla v - \int_{\Omega} f v
\ee
where $v$ is any element of $H^1(\Omega)$ that equals $\phi$ on $\pO$.
See, e.g., \cite{McleanElliptic} for a proof that this is well defined.
We repeatedly use the following lemma in various guises 
(and with little comment) throughout this paper.
\begin{lemma} Suppose $v$ and $w$ belong to $H^1(\Omega)$, $-\Lap w=f\in (H^{1}(\Omega))^*$, $v|_S=0$, and $\partial_n w|_B=0$. Then
\bes
\int_{\pO} v\;\partial_n w = 0.
\ee
\end{lemma}
\begin{proof}
Since $v|_S=0$ we have $v|_{\pO}\in H^{1/2}_{00}(B)$.  So there is a sequence $(\phi_k)$ of Lipschitz functions on $\partial\Omega$ with supports contained in $B$ converging to $v$ in $H^{1/2}(\pO)$.  By definition $\int_S \phi_k\; \partial_n w = \int_{\pO} \phi_k\; \partial_n w$.  Since $\partial_n w|_S=0$, $\int_S \phi_k \;\partial_n w = 0$.  So
$\int_{\pO} \phi_k\;\partial_n w=0$ for each $k$.  Since $\partial_n w\in H^{-1/2}(\pO) = (H^{1/2}(\pO))^*$ and since $\phi_k\ra v$ in $H^{1/2}(\pO)$
we conclude that $\int_{\pO} v\;\partial_n w = 0$.
\end{proof}

\subsection*{Equivalence of Norms}
We sketch the proofs here that the norms described in Section \ref{sec:intro}
are equivalent to the standard norms
on those spaces. In the following we
use three-barred norms $|||\cdot|||$ to denote those from Section
\ref{sec:intro} and reserve two-barred norms for their standard definitions.

Let $\psi\in H^{-1/2}(S)$ and let $u$ be the solution of 
\bes
\begin{aligned}
	-\Lap u &= 0 &\text{in $\Omega$}\\
	 \partial_n u &= \psi &\text{on $S$}\\
	u &= 0 &\text{on $B$}.
\end{aligned}
\ee
So $|||\psi|||_{H^{1/2}(S)}= ||\nabla u||_{L^2(\Omega)}$.
We see from equation \eqref{neumann-def} that 
\bes
||\partial_n u||_{H^{-1/2}(\pO)}\le ||\nabla u||_{L^2}.
\ee
By definition, 
\bes
||\partial_n u||_{H^{-1/2}(S)} \le ||\partial_n u||_{H^{-1/2}(\Omega)}
\ee
and hence $||\psi||_{H^{-1/2}(S)}\le ||\nabla u||_{L^2(\Omega)}\le ||u||_{H^1(\Omega)}$.
On the other hand, equation \eqref{mixed-cont} implies 
\bes
||u||_{H^1(\Omega)} \le c_2 ||\psi||_{H^{-1/2}(S)}
\ee
for some constant $c_2$ independent of $\psi$.  Hence
\bes
||\psi||_{H^{-1/2}(S)} \le |||\psi|||_{H^{-1/2}(S)} \le c_2 ||\psi||_{H^{-1/2}(S)}
\ee
which establishes the desired equivalence.

Let $\phi\in \Hoo(S)$ and let $u$ be the solution of 
\bes
\begin{aligned}
	-\Lap u &= 0 &\text{in $\Omega$}\\
	  u &= \phi &\text{on $\pO$}
\end{aligned}
\ee
so $|||\phi|||_{\Hoo(S)}=||\nabla u||_{L^2}^2$.
From inequality \ref{mixed-cont}, the continuity of the
trace map $H^1(\Omega)\ra H^{1/2}(\partial\Omega)$, and
the definition of the norm on $\Hoo(S)$ we see that there are
constants $c_0$ and $c_2$ such that
\bes
c_0 ||\phi||_{H^{1/2}(S)} \le ||u||_{H^{1}(\Omega)} \le c_2 ||\phi||_{H^{1/2}(S)}
\ee
But for functions in $H^1$ vanishing on $B$ it is well known that
$||u||_{H^1(\Omega)}$ is equivalent to $||\nabla u||_{L^2(\Omega)} = |||\phi|||_{\Hoo(S)}$.

Similar arguments work for the equivalence of norms in 
$H^{1/2}(S)$.  The only new ingredient is the fact that
the norm $||u||_{H^1(\Omega)}$ is equivalent to
\bes
\left\{\int_{\Omega} \abs{\nabla u}^2 + \left[\frac{1}{|S|}\int_S u\right]^2\right\}^{1/2}.
\ee

\bibliographystyle{amsalpha-abbrv}
\bibliography{KM-is-Landweber}
\end{document}